\documentclass[10pt]{article}

\usepackage[T1]{fontenc}
\usepackage[utf8]{inputenc}
\usepackage[francais,english]{babel}
\usepackage{amsmath,latexsym,amsfonts,amsthm,bm,mathtools}
\usepackage{graphicx}
\usepackage[top=1.5in, bottom=1.5in, left=1.25in, right=1.25in]{geometry}

\newtheorem{proposition}{Proposition}
\newtheorem{corollary}{Corollary}
\newtheorem{theorem}{Theorem}

\newcommand{\R}{\mathbb R}

\newcommand{\cc}{\mathbf c}

\title{The Fisher-Rao geometry of beta distributions applied to the study of canonical moments}
\author{Alice Le Brigant and St\' ephane Puechmorel}
\date{}

\begin{document}

\maketitle

\begin{abstract}
This paper studies the Fisher-Rao geometry on the parameter space of beta distributions. We derive the geodesic equations and the sectional curvature, and prove that it is negative. This leads to uniqueness for the Riemannian centroid in that space. We use this Riemannian structure to study canonical moments, an intrinsic representation of the moments of a probability distribution. Drawing on the fact that a uniform distribution in the regular moment space corresponds to a product of beta distributions in the canonical moment space, we propose a mapping from the space of canonical moments to the product beta manifold, allowing us to use the Fisher-Rao geometry of beta distributions to compare and analyze canonical moments.
\end{abstract}

\section{Introduction}

The differential geometric approach to probability theory and statistics has met increasing interest in the past years, from the theoretical point of view as well as in applications. In this approach, probability distributions are seen as elements of a differentiable manifold, on which a metric structure is defined through the choice of a Riemannian metric. Two very important ones are the Wasserstein metric, central in optimal transport, and the Fisher-Rao metric (also called Fisher information metric), essential in information geometry. Unlike optimal transport, information geometry is foremost concerned with parametric families of probability distributions, and defines a Riemannian structure on the parameter space using the Fisher information matrix \cite{fisher1922}. In parameter estimation, the Fisher information can be interpreted as the quantity of information on the unknown parameter contained in the model. As the Hessian of the well-known Kullback-Leibler divergence, it measures through the notion of curvature the capacity to distinguish between two different values of the parameter. Rao \cite{rao1992} showed that it could be used to locally define a scalar product on the space of parameters, interpretable as a Riemannian metric. An important feature of this metric is that it is invariant under any diffeomorphic change of parameterization. In fact, considering the infinite-dimensional space of probability densities on a given manifold $M$, there is a unique metric, which also goes by the name Fisher-Rao, that is invariant with respect to the action of the diffeomorphism group of $M$ \cite{cencov2000,bauer2016}. This metric induces the regular Fisher information metric on the finite dimensional submanifolds corresponding to the parameterized statistical models of interest in information geometry. Arguably the most famous example of Fisher-Rao geometry of a statistical model is that of the Gaussian model, which is hyperbolic. The multivariate Gaussian case, among other models, has also received a lot of attention \cite{atkinson1981,skovgaard1984}. 

In this work, we are interested in beta distributions, a family of probability measures on $[0,1]$ used to model random variables defined on a compact interval in a wide variety of applications. Up to our knowledge, the information geometry of beta distributions has not yet received much attention. In this paper, we give new results and properties for this geometry, and its curvature in particular. Interestingly, this geometric framework yields new by-product tools to study the set of all moments of compactly supported probability measures on the real line. This is achieved through the so-called canonical moments representation \cite{dette1997}, an alternative to the usual moment representation of a probability distribution that satisfies interesting symmetries and invariance properties.

The paper is organized as follows. Section \ref{sec:beta} deals with the study of the Fisher-Rao geometry of beta distributions. We derive the geodesic equations, prove that sectional curvature is negative, give some bounds and observe a geometrical manifestation of the central limit theorem. Section \ref{sec:moments} deals with the application to canonical moments. After a brief presentation of these objets, we propose a representation in the product beta manifold, allowing us to use the Fisher-Rao geometry of beta distributions to compare and analyze canonical moments.

\section{Geometry of the beta manifold}\label{sec:beta}

\subsection{The beta manifold}

Information geometry is concerned with parametric families of probability distributions, i.e. sets of distributions with densities with respect to a common dominant measure $\mu$ parameterized by a parameter $\theta$ member of a given set $\Theta$. That is, a collection of measures of the kind
\begin{equation*}
\mathcal P_\Theta=\{p_\theta \mu, \theta\in \Theta\}.
\end{equation*}
We assume that $\Theta$ is a non empty open subset of $\R^d$. Associated to any such family is the Fisher information matrix, defined for all $\theta$ as
\begin{equation*}
I(\theta) = \left[E\left( \frac{\partial^2}{\partial \theta_i\partial\theta_j} \ln p(X;\theta)\right)\right]_{1\leq i,j\leq d}.
\end{equation*}
As an open subset of $\R^d$, $\Theta$ is a differentiable manifold and can be equipped with a Riemannian metric using this quantity. This gives the Fisher information metric on the parameter space $\Theta$
\begin{equation*}
G^{F}_\theta(u,v) = u^t I(\theta)v, \qquad \theta\in\Theta, \quad u,v\in T_\theta\Theta \simeq \R^d,
\end{equation*}
where $u^t$ denotes the transpose of the vector $u$. By extension, we talk of the Fisher geometry of the parameterized family $\mathcal P_\Theta$, and of the Riemannian manifold $(\mathcal P_\Theta, g^F)$.

In this paper, we are interested in beta distributions, a family of probability distributions on $[0,1]$ with density with respect to the Lebesgue measure parameterized by two positive scalars $\alpha, \beta >0$
\begin{equation*}
p_{\alpha,\beta}(x) = \frac{\Gamma(\alpha+\beta)}{\Gamma(\alpha)\Gamma(\beta)} x^{\alpha-1}(1-x)^{\beta-1}, \quad x\in[0,1].
\end{equation*}
We consider the Riemannian manifold composed of the parameter space $\Theta=\R_+^*\times\R_+^*$ and the Fisher metric $g^F$, and by extension denote by beta manifold the pair $(\mathcal B, g^F)$, where $\mathcal B$ is the family of beta distributions
\begin{equation*}
\mathcal B = \{B(\alpha,\beta)=p_{\alpha,\beta}(\cdot) dx, \alpha>0, \beta>0\}.
\end{equation*}
Here $dx$ denotes the Lebesgue measure on $[0,1]$. The distance between two beta distributions is then defined as the geodesic distance associated to the Fisher metric in the parameter space
\begin{equation*}
d^{F}(B(\alpha,\beta),B(\alpha',\beta')) = \inf_{\gamma} \int_0^1 \sqrt{g^F(\dot\gamma(t),\dot\gamma(t))}dt,
\end{equation*}
where the infimum is taken over all paths $\gamma:[0,1]\rightarrow \Theta$ such that $\gamma(0)=(\alpha,\beta)$ and $\gamma(1)=(\alpha',\beta')$.

\subsection{The Fisher-Rao metric} 

The beta distributions are part of an exponential family and so the general term of the Fisher-Rao metric depends on second order derivatives of the underlying potential function. Denoting by $g^F(\alpha,\beta)$ the matrix form of $G^F_{\alpha,\beta}$,
\begin{equation}
\label{potential}
g^{F}(\alpha,\beta)= - \text{Hess}\,\varphi(\alpha, \beta),
\end{equation}
where $\varphi$ is the potential function
\begin{equation*}
\varphi (\alpha, \beta) = \ln \Gamma(\alpha) + \ln \Gamma(\beta) - \ln\Gamma(\alpha+\beta).
\end{equation*}
Proposition \ref{prop:metric} describes the metric tensor and Proposition \ref{prop:geod} the geodesic equations. 
\begin{proposition}\label{prop:metric}
The matrix representation of the Fisher-Rao metric on the space of beta distributions is given by
\begin{equation*}
g^F(\alpha,\beta)=\left[\begin{matrix} 
\psi'(\alpha) -\psi'(\alpha+\beta) & -\psi'(\alpha+\beta) \\
-\psi'(\alpha+\beta) & \psi'(\beta) -\psi'(\alpha+\beta)
\end{matrix}\right]
\end{equation*}
where $\psi$ denotes the digamma function, i.e. $\psi(x) = \frac{d}{dx}\ln\Gamma(x)$.
\end{proposition}
\begin{proof}
This follows from straightforward computations.
\end{proof}
\begin{proposition}\label{prop:geod}
The geodesic equations are given by
\begin{align*}
\ddot\alpha + a(\alpha,\beta) \dot\alpha^2 + b(\alpha,\beta)\dot\alpha\dot\beta + c(\alpha,\beta)\dot\beta^2 = 0,\\
\ddot\beta + a(\beta,\alpha) \dot\beta^2 + b(\beta,\alpha)\dot\alpha\dot\beta + c(\beta,\alpha)\dot\alpha^2 = 0,
\end{align*}
where
\begin{align*}
a(x,y) &=\frac{1}{2d(x,y)}(\psi''(x)\psi'(y) - \psi''(x)\psi'(x+y) - \psi'(y)\psi''(x+y)),\\
b(x,y) &=-\frac{1}{d(x,y)}\psi'(y)\psi''(x+y),\\
c(x,y) &=\frac{1}{2d(x,y)}(\psi''(y)\psi'(x+y) - \psi'(y)\psi''(x+y)),\\
d(x,y) &= \psi'(x)\psi'(y) - \psi'(x+y)(\psi'(x)+\psi'(y)).
\end{align*}
\end{proposition}
\begin{proof}
The geodesic equations are given by
\begin{equation}
\label{geodeq}
\begin{aligned}
\ddot \alpha + \Gamma_{\alpha\alpha}^\alpha\dot \alpha^2 + 2\Gamma_{\alpha\beta}^\alpha\dot \alpha\dot\beta + \Gamma_{\beta\beta}^\alpha \dot \beta^2=0\\
\ddot \beta + \Gamma_{\alpha\alpha}^\beta\dot \alpha^2 + 2\Gamma_{\alpha\beta}^\beta\dot \alpha\dot\beta + \Gamma_{\beta\beta}^\beta \dot\beta^2=0
\end{aligned}
\end{equation}
where the $\Gamma_{ij}^k$'s denote the Christoffel symbols of the second kind. These can be obtained from the Christoffel symbols of the first kind $\Gamma_{ij}^k$ and the coefficients $g^{ij}$ of the inverse of the metric matrix
\begin{equation*}
\Gamma_{ij}^k = \Gamma_{ijl} g^{kl}.
\end{equation*}
Here we have used the Einstein summation convention. Since the Fisher metric is a Hessian metric, the Christoffel symbols of the first kind can be obtained as
\begin{equation*}
\Gamma_{ijk} = \frac{1}{2}\varphi_{ijk},
\end{equation*}
where $\varphi$ is the potential function \eqref{potential}. Straightforward computation yields the desired equations.
\end{proof}
Notice that when $\alpha=\beta=\gamma$, both geodesic equations \eqref{geodeq} yield a unique ordinary differential equation
\begin{equation*}
\ddot \gamma + (a(\gamma,\gamma)+b(\gamma,\gamma)+c(\gamma,\gamma)) \dot\gamma^2 = 0.
\end{equation*}
The line of equation $\alpha=\beta$ is therefore a geodesic for the Fisher metric. More precisely, we have the following corollary obtained directly from Proposition \ref{prop:geod}.
\begin{corollary}
The line of equation $\alpha(t)=\beta(t)=\gamma(t)$, where
\begin{equation*}
\ddot\gamma + \frac{\psi'(\gamma)\psi''(\gamma) - 4\psi'(\gamma)\psi''(2\gamma)}{2(\psi'(\gamma)^2 - 2\psi'(\gamma)\psi'(2\gamma))} \dot\gamma^2 = 0,
\end{equation*}
is a geodesic for the Fisher metric.
\end{corollary}

\begin{figure}
\centering
\includegraphics[width=0.49\textwidth]{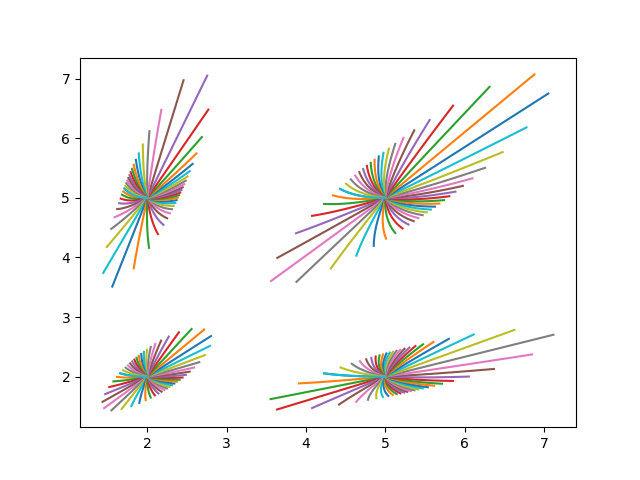}
\includegraphics[width=0.49\textwidth]{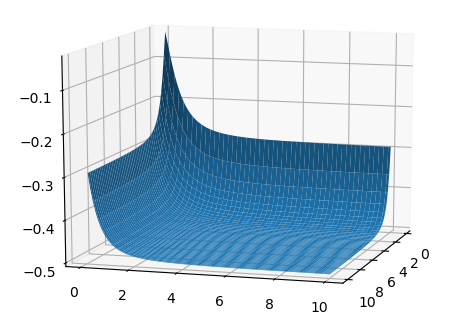}
\caption{Geodesic balls (left) and sectional curvature (right) of the beta manifold.}
\end{figure}

\subsection{Some properties of the polygamma functions}

In order to further study the geometry of the beta manifold, we will need a few technical results on the polygamma functions. The polygamma functions are the successive derivatives of the logarithm of the Euler Gamma function $\Gamma(x)$, i.e.
\begin{equation*}
\psi^{(m-1)}(x) := \frac{d^m}{dx^m}\ln \Gamma(x), \quad m\geq 1.
\end{equation*}
Their series representation is given by:
\begin{equation*}
\psi^{(m)} = (-1)^{m+1}m!\sum_{k\geq 0} \frac{1}{(k+x)^{m+1}}.
\end{equation*}
In the sequel, we are mostly interested by the first three, i.e. 
\begin{equation*}
\psi'(x) = \sum_{k\geq 0}\frac{1}{(k+x)^2}, \quad \psi''(x)=-2\sum_{k\geq 0}\frac{1}{(k+x)^3}, \quad \psi'''(x)=6\sum_{k\geq0}\frac{1}{(k+x)^4},
\end{equation*}
and we will use the following equivalents in the neighborhood of zero, given by the first term of their series
\begin{equation}
\label{equivzero}
\psi'(x) \underset{x\to 0}{\sim} \frac{1}{x^2}, \quad \psi''(x) \underset{x\to 0}{\sim} -\frac{2}{x^3}, \quad \psi'''(x) \underset{x\to0}{\sim}\frac{6}{x^4}.
\end{equation}
In the neighborhood of infinity, we will need the following expansions
\begin{equation}
\label{equivinfty}
\begin{aligned}
\psi(x) \underset{x\to+\infty}{=} \ln(x) -\frac{1}{2x}+ o\bigg(\frac{1}{x^2}\bigg),\\
\psi'(x) \underset{x\to+\infty}{=} \frac{1}{x} + \frac{1}{2x^2} + o\bigg(\frac{1}{x^2}\bigg),\\
\psi''(x) \underset{x\to+\infty}{=} -\frac{1}{x^2} - \frac{1}{x^3} + o\bigg(\frac{1}{x^3}\bigg).
\end{aligned}
\end{equation}

\subsection{Curvature of the Fisher-Rao metric}

In this section, we prove our main result, that is that the sectional curvature of the beta manifold is negative.
\begin{proposition}
The sectional curvature of the Fisher metric is given by:
\begin{equation*}
K(\alpha,\beta) = \frac{\psi''(\alpha)\psi''(\beta)\psi''(\alpha+\beta)}{4\,d(\alpha, \beta)^2}\bigg(\frac{\psi'(\alpha)}{\psi''(\alpha)} + \frac{\psi'(\beta)}{\psi''(\beta)} - \frac{\psi'(\alpha+\beta)}{\psi''(\alpha+\beta)}\bigg),
\end{equation*}
\end{proposition}
\begin{proof}
The sectional curvature of a Hessian metric is given by
\begin{equation*}
K = \frac{1}{4(\det g)^2} R_{1212}
\end{equation*}
where
\begin{equation*}
R_{1212} = -\varphi_{\beta\beta} (\varphi_{\alpha\alpha\alpha}\varphi_{\alpha\beta\beta} - \varphi_{\alpha\alpha\beta}^2)
+ \varphi_{\alpha\beta} (\varphi_{\alpha\alpha\alpha}\varphi_{\beta\beta\beta} - \varphi_{\alpha\alpha\beta}\varphi_{\alpha\beta\beta})
+ \varphi_{\alpha\alpha} (\varphi_{\alpha\alpha\beta}\varphi_{\beta\beta\beta} - \varphi_{\alpha\beta\beta}^2).
\end{equation*}
Computing the partial derivatives of the potential function $\varphi$ gives
\begin{align*}
\varphi_{\alpha\alpha\alpha} &= \psi''(\alpha) - \psi''(\alpha+\beta),  \\
\varphi_{\beta\beta\beta} &= \psi''(\beta) - \psi''(\alpha + \beta), \\
\varphi_{\alpha\alpha\beta} &= \varphi_{\alpha\beta\beta} = -\psi''(\alpha+\beta),
\end{align*}
and the determinant of the metric is given by
\begin{equation*}
\det g(\alpha,\beta) = \psi'(\alpha)\psi'(\beta) - \psi'(\alpha+\beta)(\psi'(\alpha)+\psi'(\beta)).
\end{equation*}
This gives
\begin{equation*}
K = \frac{\psi''(\alpha+\beta)(\psi'(\alpha)\psi''(\beta)+\psi''(\alpha)\psi'(\beta)) - \psi'(\alpha+\beta)\psi''(\alpha)\psi''(\beta)}{4(d(\alpha,\beta))^2}.
\end{equation*}
Factorizing the numerator by $\psi''(\alpha)\psi''(\beta)\psi''(\alpha+\beta)$ yields the desired result.
\end{proof}

\begin{proposition}
\label{prop:asymptotic}
The asymptotic behavior of the sectional curvature is given by
\begin{align*}
\lim_{\beta\to 0}K(\alpha, \beta) &= \lim_{\beta\to 0}K(\beta, \alpha) = \frac{3}{4} - \frac{\psi'(\alpha)\psi'''(\alpha)}{2\, \psi''(\alpha)^2},\\
\lim_{\beta\to \infty} K(\alpha, \beta) &= \lim_{\beta\to \infty} K(\beta, \alpha) = \frac{\alpha\, \psi''(\alpha) + \psi'(\alpha)}{4(\alpha\,\psi'(\alpha) - 1)^2}.
\end{align*}
Moreover, we have the following limits
\begin{align*}
\lim_{\alpha,\beta\to 0} K(\alpha,\beta) = 0, \quad \lim_{\alpha, \beta\to \infty}K(\alpha,\beta) = -\frac{1}{2},\\
\lim_{\alpha\to 0, \beta\to \infty} K(\alpha, \beta)=\lim_{\alpha\to \infty, \beta\to 0} K(\alpha, \beta) = -\frac{1}{4}.
\end{align*}
\end{proposition}
\begin{proof}
Let us fix $\alpha\in\R_+^*$, and denote $x=\beta$ the varying parameter of the beta distribution. 
The asymptotic behavior of the sectional curvature can be obtained by separately examining its numerator and the metric determinant appearing at the denominator
\begin{equation*}
K(\alpha,x) = \frac{N(\alpha,x)}{4d(\alpha,x)^2}.
\end{equation*}
Using a first order Taylor development of $\psi'$ in $\alpha$ and the equivalent \eqref{equivzero}, we deduce the following expansion for the determinant around zero
\begin{equation*}
d(\alpha,x) = \psi'_x(\psi'_\alpha-\psi'_{\alpha+x}) - \psi'_{a+x}\psi'_a \underset{x\to 0}{=} -\frac{\psi''_\alpha}{x} + o\bigg(\frac{1}{x}\bigg).
\end{equation*}
Similarly, writing the numerator of the sectional curvature as
\begin{align*}
N(\alpha, x) &:= \psi''_{\alpha+x}(\psi''_{\alpha}\psi'_{x} +\psi'_{\alpha}\psi''_{x}) - \psi'_{\alpha+x}\psi''_{\alpha}\psi''_{x}\\
&= (\psi''_{\alpha+x} - \psi''_{\alpha})(\psi''_{\alpha}\psi'_{x}+\psi'_{\alpha}\psi''_{x}) + (\psi'_a-\psi'_{a+x})\psi''_\alpha\psi''_x + \psi''_{\alpha}(\psi''_{\alpha}\psi'_{x}+\psi'_{\alpha}\psi''_{x}) - \psi'_a\psi''_\alpha\psi''_x,
\end{align*}
we get the following behavior around zero
\begin{align*}
N(\alpha,x)&\underset{x\to0}{=} x\psi'''_\alpha\bigg(\frac{\psi''_\alpha}{x^2} - \frac{2\psi'(\alpha)}{x^3}\bigg) + x(\psi''_\alpha)^2\frac{2}{x^3} + \psi''_\alpha\bigg(\frac{\psi''_\alpha}{x^2} - \frac{2\psi'_\alpha}{x^3}\bigg) + \frac{2\psi'_\alpha\psi''_\alpha}{x^3} + o\bigg(\frac{1}{x^2}\bigg)\\
&\underset{x\to0}{=} \frac{3(\psi''_\alpha)^2 - 2\psi'_\alpha\psi'''_\alpha}{x^2} + o\bigg(\frac{1}{x^2}\bigg).
\end{align*}
This yields the desired expression for the limit of $K(\alpha, x)$ as $x\to 0$. Now, in the neighborhood of infinity, the expansions \eqref{equivinfty} yield the following behavior for the determinant
\begin{align*}
d(\alpha, x) &= \psi'_\alpha\psi'_\beta - \psi'_{\alpha+\beta}(\psi'_\alpha+\psi'_\beta) \\
&\underset{x\to+\infty}{=} \psi'_\alpha\bigg(\frac{1}{x}+\frac{1}{2x^2}\bigg) - \bigg(\frac{1}{\alpha+x}+\frac{1}{2(\alpha+x)^2}\bigg)\bigg(\psi'_\alpha + \frac{1}{x}\bigg) + o\bigg(\frac{1}{x^2}\bigg)\\
&\underset{x\to+\infty}{=} \frac{\alpha\psi'_\alpha - 1}{x^2} + o\bigg(\frac{1}{x^2}\bigg),
\end{align*}
while an expansion of the numerator gives
\begin{align*}
N(\alpha,x) &\underset{x\to+\infty}{=} -\bigg(\frac{1}{(a+x)^2} + \frac{1}{(a+x)^3}\bigg)\bigg(\frac{\psi''_\alpha}{x} + \frac{\psi''_\alpha}{2x^2} - \frac{\psi'_\alpha}{x^2}\bigg) \\
&\hspace{8em}+ \bigg(\frac{1}{a+x} + \frac{1}{2(a+x)^2}\bigg)\psi''_\alpha\bigg(\frac{1}{x^2}+\frac{1}{x^4}\bigg) + o\bigg(\frac{1}{x^3}\bigg)\\
&\underset{x\to+\infty}{=} \frac{\alpha\psi''_\alpha + \psi'_\alpha}{x^4} + o\bigg(\frac{1}{x^4}\bigg),
\end{align*}
yielding again the desired limit for $K$. Finally, approximating $\psi'''_x$ by $6/x^4$ when $x\to0$, we get
\begin{align*}
\lim_{\alpha\to 0, \beta\to +\infty} K(\alpha, \beta) &= \lim_{\alpha\to0} \frac{\alpha\, \psi''_\alpha + \psi'_\alpha}{4(\alpha\,\psi'_\alpha - 1)^2} = \lim_{\alpha\to0}\frac{-1/\alpha^2}{4/\alpha^2} = -\frac{1}{4},\\
\lim_{\alpha,\beta\to +\infty} K(\alpha, \beta) = &\lim_{\alpha,\beta\to+\infty} \frac{-1/(2\alpha^2)}{1/\alpha^2}=-\frac{1}{2},\\
\lim_{\alpha,\beta\to 0} K(\alpha,\beta) =& \frac{3}{4} - \frac{1}{2} \lim_{\alpha,\beta\to 0} \frac{6/\alpha^6}{(-2/\alpha^3)^2} = 0,
\end{align*}
which completes the proof.
\end{proof}

We can now show the following property.
\begin{proposition}
The sectional curvature is negative and bounded from below.
\end{proposition}
\begin{proof}
Recall that in its most factorized form, the sectional curvature is given by
\begin{equation*}
K(\alpha,\beta) = \frac{\psi''(\alpha)\psi''(\beta)\psi''(\alpha+\beta)}{4\,d(\alpha, \beta)^2}\bigg(\frac{\psi'(\alpha)}{\psi''(\alpha)} + \frac{\psi'(\beta)}{\psi''(\beta)} - \frac{\psi'(\alpha+\beta)}{\psi''(\alpha+\beta)}\bigg),
\end{equation*}
Since $\psi''$ is negative, the first factor is negative and so there remains to prove that the function $x\mapsto \frac{\psi'(x)}{\psi''(x)}$ is sub-additive, i.e.
\begin{equation*}
\frac{\psi'(\alpha)}{\psi''(\alpha)} + \frac{\psi'(\beta)}{\psi''(\beta)} - \frac{\psi'(\alpha+\beta)}{\psi''(\alpha+\beta)} \geq 0.
\end{equation*}
This has been shown recently in \cite{yang2017} (Corollary 4). Now, to show that it is bounded from below, set
\begin{align*}
k_1(\alpha) := \lim_{\beta\to 0}K(\alpha, \beta) = \frac{3}{4} - \frac{\psi'(\alpha)\psi'''(\alpha)}{2\, \psi''(\alpha)^2},\\
k_2(\alpha) := \lim_{\beta\to +\infty} K(\alpha, \beta) = \frac{\alpha\, \psi''(\alpha) + \psi'(\alpha)}{4(\alpha\,\psi'(\alpha) - 1)^2}.
\end{align*}
$k_1$ and $k_2$ are continuous functions on $\R_+^*$, and according to Proposition \ref{prop:asymptotic} they have finite limits at the boundaries
\begin{align*}
\lim_{\alpha\to 0} k_1(\alpha) = 0, \quad \lim_{\alpha\to+\infty} k_1(\alpha) = -\frac{1}{4},\\
\lim_{\alpha\to 0} k_2(\alpha) = -\frac{1}{4}, \quad \lim_{\alpha\to+\infty} k_2(\alpha) = -\frac{1}{2}.
\end{align*}
Therefore, they are bounded, i.e., there exist negative finite constants $M_1$ and $M_2$ such that for all $\alpha\in \R_+^*$,
\begin{equation*}
\lim_{\beta\to 0} K(\alpha, \beta) > M_1, \quad \lim_{\beta\to+\infty} K(\alpha, \beta) > M_2.
\end{equation*}
Setting $f(\beta):=\inf_{\alpha\in\R_+^*}K(\alpha,\beta)$, notice that $f$ is a continuous function on $\R_+^*$ due to the continuity of $K$ in both its variables and the invertibility of the limit and infimum. For this last reason, we also obtain
\begin{align*}
\lim_{\beta\to0}\inf_{\alpha\in\R_+^*}K(\alpha,\beta)=\inf_{\alpha\in\R_+^*}\lim_{\beta\to0}K(\alpha,\beta) > M_1,\\
\lim_{\beta\to+\infty}\inf_{\alpha\in\R_+^*}K(\alpha,\beta)=\inf_{\alpha\in\R_+^*}\lim_{\beta\to+\infty}K(\alpha,\beta) > M_2,
\end{align*}
i.e., $f$ has finite limits at the boundaries and is therefore bounded, in particular from below
\begin{equation*}
\inf_{\beta\in\R_+^*}\inf_{\alpha\in\R_+^*}K(\alpha, \beta) > -\infty.
\end{equation*}
\end{proof}
The fact that the beta manifold has negative curvature is particularly interesting for the computation of Riemannian centroids such as Fréchet or Karcher means \cite{frechet1948, karcher1977}. In general, the Fréchet mean on a Riemannian manifold is not unique. However, when the curvature is negative, there is no cut locus and uniqueness holds. In this context, it is defined for any given sequence of probability measure $B_1,\ldots, B_n$ as
\begin{equation*}
\bar B = \underset{B\in \mathcal B}{\text{argmin}} \sum_{i=1}^n d^F(B,B_i)^2.
\end{equation*}
This quantity can be computed using a gradient descent algorithm the Karcher flow algorithm. 

\subsection{A lower bound on the determinant of the metric}
The determinant of the metric is the key ingredient to volume computations. In this section, a lower bound of this determinant is computed, which is also its asymptotic value.
\begin{proposition}
\label{prop:detg_integral}
The determinant of the information metric matrix admits the following integral representation:
\begin{equation}
    \label{eq:detg_integral}
    \left|g(\alpha,\beta)\right| = \int_{\R^+}\int_0^1 
    \frac{x (1-x)}{(1-e^{-t x})(1-e^{-t(1-x)})} \left(
    \left(e^{\beta t x} -1 \right)\left(e^{-\alpha t (1-x)}\right) -1 
    \right) e^{-(\alpha + \beta)t} dx dt
\end{equation}
\end{proposition}
\begin{proof}
The polygamma function of order $n$ can be expressed as an integral \cite{olver2010nist} :
\begin{equation}
    \label{eq:polygamma_int}
    \psi^{(n)}(x) = (-1)^{n+1}\int_{\R^+} \frac{t^n}{1-e^{-t}} e^{-x t} dt, \, x > 0
\end{equation}
The determinant $ \left|g(\alpha,\beta)\right|$ expands as:
\begin{equation}
    \label{eq:detg_expansion}
    \left|g(\alpha,\beta)\right| = \left(\psi^\prime(\alpha)-\psi^\prime(\alpha+\beta)\right)
    \left(\psi^\prime(\beta)-\psi^\prime(\alpha+\beta)\right)
    - \psi^\prime(\alpha+\beta)^2
\end{equation}
Using the integral \ref{eq:polygamma_int}, it comes:
\begin{equation}
    \label{eq:psiprime_reduction}
    \psi^\prime(\alpha)- \psi^\prime(\alpha+\beta) = \int_{\R^+} \frac{t}{1-e^{-t}} e^{-\alpha t} - e^{-(\alpha+\beta)t}dt = 
    \int_{\R^+} \frac{t}{1-e^{-t}} (e^{\alpha t} - 1 )e^{-(\alpha+\beta)t}dt
\end{equation}
The difference $\psi^\prime(\alpha)- \psi^\prime(\alpha+\beta)$ is thus equal to the laplace transform at $\alpha +\beta$ of the function:
\begin{equation}
    \label{eq:diff_laplace}
    \frac{t}{1-e^{-t}} (e^{\alpha t} - 1 )
\end{equation}
Using the convolution theorem \cite{schiff2013laplace}, it comes:
\begin{equation}
\label{eq:laplace_conv_diff}
\begin{split}
    & \left(\psi^\prime(\alpha)-\psi^\prime(\alpha+\beta)\right)
    \left(\psi^\prime(\beta)-\psi^\prime(\alpha+\beta)\right) = \\
    & \int_{\R^+} \left(\int_0^t\frac{x(t-x)}{(1-e^{-x})(1-e^{-(1-x)})} (e^{\beta x} - 1 )(e^{\alpha(t-x)}-1)dx\right) e^{-(\alpha+\beta)t}dt = \\
   & \int_{\R^+} t^3 \left(\int_0^1\frac{x(1-x)}{(1-e^{-t x})(1-e^{-t (1-x)})} (e^{\beta t x} - 1 )(e^{\alpha t (1-x)}-1)dx\right) e^{-(\alpha+\beta)t}dt
\end{split}
\end{equation}
The same procedure can be applied to the integral expression of $\psi^\prime(\alpha+\beta)$ to obtain:
\begin{equation}
    \label{eq:laplace_conv_psi2}
    \begin{split}
        & \psi^{\prime 2}(\alpha+\beta) =  \\
        & \int_{\R^+} t^3 \left(\int_0^1\frac{x(1-x)}{(1-e^{-t x})(1-e^{-t (1-x)})}dx\right) e^{-(\alpha+\beta)t}dt
    \end{split}
\end{equation}
Combining \ref{eq:laplace_conv_diff} and \ref{eq:laplace_conv_psi2} gives:
\begin{equation}
    \label{eq:detg_integral}
    \begin{split}
        & |g(\alpha,\beta)| = \\
         & \int_{\R^+} t^3 \left(\int_0^1\frac{x(1-x)}{(1-e^{-t x})(1-e^{-t (1-x)})} \left[ (e^{\beta t x} - 1 )(e^{\alpha t (1-x)}-1)-1\right] dx\right) e^{-(\alpha+\beta)t}dt
    \end{split}
\end{equation}
\end{proof}
Building on the integral representation of Proposition \ref{prop:detg_integral}, it is possible to derive a lower bound for the determinant, which is also its asymptotic value.
\begin{proposition}
The following lower bound holds:
\begin{equation}
    \label{eq:detg_lower_bound}
     |g(\alpha,\beta)| > \frac{1 + \alpha + \beta}{2\alpha \beta\left(\alpha+\beta\right)^2}
\end{equation}
\end{proposition}
\begin{proof}
The hyperbolic cotangent $\coth$ satisfies:
\begin{equation}
    \label{eq:coth_half}
    \coth{\frac{x}{2}} = \frac{1+e^{-x}}{1-e^{-x}}
\end{equation}
and so:
\begin{equation}
    \label{eq:coth_expansion}
    \frac{1}{1-e^{-x}} = \frac{1}{2} + \frac{1}{2} \coth{\frac{x}{2}}
\end{equation}
Letting:
\begin{equation}
    \label{eq:inner_integrand}
    h(x,t) = \frac{x}{2} \left(1+\coth{\frac{x}{2}}\right)
\end{equation}
The integral expression \ref{eq:detg_integral} is rewritten as:
\begin{equation}
    \label{eq:detg_integral_coth}
    \begin{split}
        & |g(\alpha,\beta)| = \\
         & \int_{\R^+} t^3 \left(\int_0^1 h(x,t)h\left((1-x)t\right) \left[ (e^{\beta t x} - 1 )(e^{\alpha t (1-x)}-1)-1\right] dx\right) e^{-(\alpha+\beta)t}dt
    \end{split}
\end{equation}
Since:
\[
\coth{x} > \frac{1}{x}, \, x > 0
\]
it comes:
\begin{equation}
    \label{eq:h_bound}
    \begin{split}
    h(x,t)h\left((1-x)t\right) & > \frac{x(1-x)}{4}\left(1+\frac{2}{tx}\right)\left(1+\frac{2}{t(1-x)}\right) \\
    & > \frac{1}{4} \left(x + \frac{2}{t}\right)\left(1-x + \frac{2}{t}\right) 
    \end{split}
\end{equation}
a lower bound for \ref{eq:detg_integral_coth} is thus given by:
\begin{equation}
    \label{eq:detg_integral_bound}
    \begin{split}
        & |g(\alpha,\beta)| > \\
         & \int_{\R^+} t^3 \left(\int_0^1 \frac{1}{4} \left(x + \frac{2}{t}\right)\left(1-x + \frac{2}{t}\right)  \left[ (e^{\beta t x} - 1 )(e^{\alpha t (1-x)}-1)-1\right] dx\right) e^{-(\alpha+\beta)t}dt
    \end{split}
\end{equation}
The inner term:
\[
I(t) = t^3 \int_0^1 \frac{1}{4} \left(x + \frac{2}{t}\right)\left(1-x + \frac{2}{t}\right)  \left[ (e^{\beta t x} - 1 )(e^{\alpha t (1-x)}-1)-1\right] dx
\]
admits a closed form expression:
\begin{equation}
    \label{eq:inner_closed_form}
    I(t)  = \frac{1}{4 a^3 b^3 (a-b)^3}I_1(t) + I_2(t) -I_3(t) - I_4(t) +I_5(t) - I_6(t)
\end{equation}
with:
\begin{align}
    \label{eq:inner_closed_form_terms}
    & I_1(t)= a^6 \left(-\left(2 b^2 (t+2) \left(e^{b t}-1\right)+b t \left(e^{b t}+1\right)-2 e^{b t}+2\right)\right) \\
    & I_2(t) = a^5 b \left(4 b^2 (t+2) \left(e^{b t}-1\right)+3 b t \left(e^{b t}+1\right)-6 e^{b t}+6\right) \\
    & I_3(t) = 2 a^4 b^2 \left(b^2 (t+2) \left(-\left(e^{a t}-e^{b t}\right)\right)+b t \left(e^{b t}+2\right)-3 e^{b t}+3\right) \\
    & I_4(t) = 2 a^3 b^4 \left(2 b (t+2) \left(e^{a t}-1\right)-t \left(e^{a t}+2\right)\right) \\
    & I_5(t) = a^2 b^4 \left(2 b^2 (t+2) \left(e^{a t}-1\right)-3 b t \left(e^{a t}+1\right)-6 e^{a t}+6\right) \\
    & I_6(t) = 2 b^6 \left(e^{a t}-1\right)+a b^5 \left(e^{a t} (b t+6)+b t-6\right)
\end{align}
Performing the outer integration yields finally:
\begin{equation}
    |g(\alpha,\beta)| > \frac{1 + \alpha + \beta}{2\alpha \beta\left(\alpha+\beta\right)^2}
\end{equation}
thus completing the proof.
\end{proof}

\subsection{A geometric view point of the central limit theorem}

The central limit theorem tells us that once re-centered, a beta distribution $B(n\alpha,n\beta)$ converges at rate $\sqrt{n}$ to a centered normal distribution
\begin{equation*}
\sqrt{n}\left(B(n\alpha,n\beta) - \frac{\alpha}{\alpha+\beta}\right) \underset{n\to\infty}{\rightarrow} \mathcal N\left(0, \frac{ab}{(a+b)^3}\right).
\end{equation*}
For a fixed $\lambda >0$, the line $\beta= \lambda \alpha$ corresponds to all the beta distributions of mean $1/(1+\lambda)$. Asymptotically, we retrieve a hyperbolic distance between two distributions on this line.
\begin{proposition}
When $\beta = \lambda \alpha$ for a fixed $\lambda >0$, the metric is asymptotically 
\begin{equation*}
ds^2 = \frac{d\alpha^2}{2\alpha^2} + o\left(\frac{1}{\alpha^2}\right).
\end{equation*}
This means
\begin{equation*}
d^F(B(n\alpha, n\lambda\alpha), B(n\alpha',n\lambda\alpha')) \underset{n\to\infty}{\rightarrow} d^F(\mathcal N(0, \alpha), \mathcal N(0, \alpha'))
\end{equation*}
\end{proposition} 

\begin{proof}
The infinitesimal element of length is given by
\begin{equation*}
ds^2 = (\psi'(\alpha)-\psi'(\alpha+\beta))d\alpha^2+(\psi'(\beta)-\psi'(\alpha+\beta))d\beta^2 - 2\psi'(\alpha+\beta) d\alpha d\beta,
\end{equation*}
and so when $\beta = \lambda \alpha$ for a fixed $\lambda >0$,
\begin{equation*}
ds^2 = G(\alpha)d\alpha^2
\end{equation*}
where
\begin{equation*}
G(\alpha)=\psi'(\alpha)+\lambda^2\psi'(\lambda\alpha)-(1+\lambda)^2\psi'((1+\lambda)\alpha).
\end{equation*}
When $\alpha\to\infty$, we have asymptotically using \eqref{equivinfty}
\begin{equation*}
G(\alpha) = \frac{1}{\alpha} + \frac{1}{2\alpha^2} + \frac{\lambda}{\alpha} + \frac{1}{2\alpha^2} - \frac{1+\lambda}{\alpha} - \frac{1}{2\alpha^2} + o\left(\frac{1}{\alpha^2}\right) = \frac{1}{2\alpha^2} + o\left(\frac{1}{\alpha^2}\right),
\end{equation*}
yielding the desired result.
\end{proof}

\section{Geometric structure of canonical moments}\label{sec:moments}

In this section we use the previously described geometry to equip the space of canonical moments with a natural geometric structure. Indeed, beta distributions naturally arise in the study of these quantities. The canonical moments of a probability distribution offer an alternative to the usual moment representation, that can be considered more intrinsic. Qualitative properties such as symmetry or the fact that two probability distributions are identical up to linear transformation are easier seen on canonical moments than on regular moments. We first start by recalling the definition and some useful properties of the canonical moments. For an overview on the subject, we refer the reader to the monograph \cite{dette1997}.

\subsection{Canonical moments: definition and useful properties}

Here we are interested in probability distributions defined on finite intervals $[a,b]$ with $a<b$. Since any such distribution can be uniquely mapped to one on $[0,1]$ through translation and rescaling $x\mapsto (x-a)/(b-a)$, we restrict to the cas $[a,b]=[0,1]$. Let $\mathcal P$ denote the set of all probability measures on $[0,1]$, and $\mu\in\mathcal P$ be such a measure. Its $n$-th order moment is given by
\begin{equation*}
c_n(\mu)=\int_a^b x^n d\mu(x).
\end{equation*}
Following \cite{dette1997}, we respectively denote its infinite sequence of moments and the truncated sequence of its $n$ first moments by
\begin{align*}
\cc(\mu)&=(c_1(\mu),c_2(\mu),\hdots)\\
\cc_n(\mu)&=(c_1(\mu), \hdots, c_n(\mu))
\end{align*}  
The spaces of all moment sequences and sequences of size $n$ corresponding to the $n$ first moments of a probability measure are respectively denoted by
\begin{equation*}
M = \{\cc(\mu), \mu\in \mathcal P\}, \qquad M_n = \{ \cc_n=(c_1,\hdots, c_n)\in[0,1]^n,\exists \mu\in \mathcal P\,\, \cc_n=\cc_n(\mu)\}.
\end{equation*}
Obvious elements of $M_n$ are the $n$-tuples $\cc_n(x)=(x,\hdots,x^n)$ corresponding to the first $n$ moments of Dirac distributions at points $x\in [0,1]$. In fact, it can be shown that $M_n$ is the convex hull of the curve $\cc_n(x)=(x,\hdots,x^n)$ of $[0,1]^n$ \cite[Theorem~1.2.1.]{dette1997}.
For $n=2$, this yields the area contained in-between the diagonal of the square $[0,1]^2$ and the parabola $x\mapsto x^2$, as shown in Figure \ref{fig:momentspace}. The points of the diagonal with coordinates $(x,x)$ uniquely correspond to the first two moments of Bernoulli distributions of parameter $x\in[0,1]$, while the points of the parabola uniquely represent Dirac distributions. More generally, any point of the boundary of $M_n$ corresponds to the first moments of a unique probability measure, which is discrete (i.e., a linear combination of Dirac distributions), whereas any point of the interior corresponds to an infinity of probability distributions. In other words, for each sequence $\cc_n=(c_1,\hdots,c_n)\in M_n$, the set of measures
\begin{equation*}
\mathcal P(\cc_n) = \{ \mu\in \mathcal P, \cc_n(\mu) = \cc_n\},
\end{equation*}
is infinite if $\cc_n\in \text{Int} M$ and it is a singleton if $\cc_n\in\partial M_n$. For a fixed $\cc_{n-1}\in M_{n-1}$, the intersection of the vertical line above $\cc_{n-1}$ and the compact and convex set $M_n$ in $[0,1]^n$ yields an interval of all possible values for the following $n^{\text{th}}$ moment, with minimum and maximum values
\begin{equation*}
c_{n}^- = \min_{\mu\in\mathcal P(\cc_{n-1})} c_{n}(\mu), \qquad c_{n}^+ = \max_{\mu\in\mathcal P(\cc_{n-1})} c_{n}(\mu).
\end{equation*}
Note that $c_{n}^+$ and $c_{n}^-$ are equal if $\cc_{n-1}$ is a boundary point. For $\cc\in M$, let $N(\cc)=\min\{n\in\mathbb N, \cc_n\in\partial M_n\}$. Then for all $n\leq N(\cc)$, the $n^{\text{th}}$ canonical moment describes the relative position of the $n^{\text{th}}$ regular moment $c_n$ with respect to these lower and upper bounds
\begin{equation*}
p_n = \frac{c_n - c_n^-}{c_n^+ - c_n^-}.
\end{equation*}
\begin{figure}
\centering
\includegraphics[width=20em]{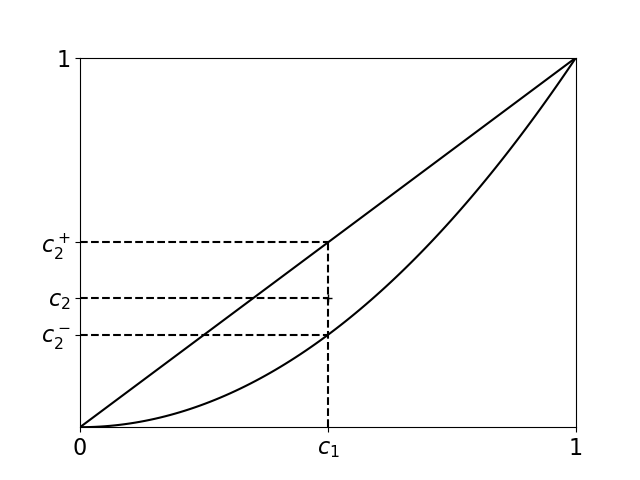}
\caption{The space $M_2$ of all first two moments on $[0,1]$.}
\label{fig:momentspace}
\end{figure}
The canonical moments provide an alternative representation of the underlying distribution that can be seen as more intrinsic in some sense. Indeed, the canonical moments representation benefits from the following interesting properties  \cite[Theorem 1.3.2 and Corollary 1.3.4]{dette1997}.
\begin{theorem}
Let $\mu\in \mathcal P$ and $a<b$. We denote by $\mu_{ab}$ the image measure of $\mu$ by the linear transformation $x\mapsto (b-a)x+a$, by $\mu^{(r)}$ the reflection of $\mu$ with respect to the point $1/2$. Then we have the following properties.
\begin{itemize}
\item[(i)] The canonical moments of $\mu$ remain invariant under linear transformation, i.e.
\begin{equation*}
p_k(\mu) = p_k(\mu_{ab}), \quad k=1,\hdots, N(\cc).
\end{equation*}
\item[(ii)] The measure $\mu$ is symmetric, i.e. $\mu=\mu^{(r)}$, if and only if
\begin{equation*}
p_k=1/2, \quad 1\leq k\leq (N(\cc)+1)/2.
\end{equation*}
\end{itemize}
\end{theorem}
The canonical moments can be expressed using the so-called Hankel determinants, defined for any $n\geq 0$ by
\begin{align*}
\underline{H}_{2k} = \left| 
\begin{matrix} 
c_0 &\hdots& c_k \\
\vdots & \ddots & \vdots \\
c_k & \hdots & c_{2k}
\end{matrix}\right|, \qquad
\overline{H}_{2k} = \left| 
\begin{matrix} 
c_1-c_2 &\hdots& c_k - c_{k+1} \\
\vdots & \ddots & \vdots \\
c_k - c_{k+1} & \hdots & c_{2k-1} - c_{2k}
\end{matrix}\right| \\
\underline{H}_{2k+1}= \left| 
\begin{matrix} 
c_0 &\hdots& c_{k+1} \\
\vdots & \ddots & \vdots \\
c_{k+1} & \hdots & c_{2k+1}
\end{matrix}\right|, \qquad
\overline{H}_{2k+1} = \left| 
\begin{matrix} 
c_0-c_1 &\hdots& c_k - c_{k+1} \\
\vdots & \ddots & \vdots \\
c_k - c_{k+1} & \hdots & c_{2k} - c_{2k+1}
\end{matrix}\right| 
\end{align*}
The Hankel determinants verify the following relation \cite[Theorem 1.4.5]{dette1997}.
\begin{theorem}
For all $n\geq 0$,
\begin{equation*}
\underline{H}_n\overline{H}_n = \underline{H}_{n-1}\overline{H}_{n+1} + \overline{H}_{n-1}\underline{H}_{n+1}
\end{equation*}
\end{theorem}
The relative position of the $n^{th}$ regular moment with respect to the lower and upper bounds $c_n^-$ and $c_n^+$ can be expressed in terms of quotients of Hankel determinants \cite[Theorem 1.4.4]{dette1997}
\begin{theorem}
If $\cc_{n-1} \in \text{Int}\, M_{n-1}$ then
\begin{equation*}
c_n - c_n^- = \underline{H}_n /\underline{H}_{n-2}, \qquad c_n^+ - c_n = \overline{H}_n/\overline{H}_{n-2}.
\end{equation*}
\end{theorem}
Using these two results, it is possible to show that the Lebesgue measure in the interior of the moment space corresponds to a product measure of unnormalized beta distributions. More precisely, choosing an element $\cc_n=(c_1,\hdots,c_n)$ at random in the interior of $M_n$ is equivalent to choosing each canonical moment $p_k$ independently in $]0,1[$ according to a beta distribution $B(n-k+1,n-k+1)$ for $k=1,\hdots,n$. This is expressed by the following result, which corresponds to \cite[Example 1.4.12]{dette1997}, and for which we give the proof here for the sake of completeness. 
\begin{theorem}
\label{thm:betamoments}
\begin{equation*}
dc_1\hdots dc_n = \prod_{k=1}^n p_k^{n-k}(1-p_k)^{n-k} dp_k.
\end{equation*}
\end{theorem}
\begin{proof}
Using the two previous results, the $n^{th}$ canonical moment $p_n$ and $q_n=1-p_n$ can be expressed as
\begin{equation*}
p_n = \frac{\underline{H}_n\overline{H}_{n-2}}{\underline{H}_{n-1}\overline{H}_{n-1}}, \qquad q_n = 1-p_n =  \frac{\overline{H}_n\underline{H}_{n-2}}{\underline{H}_{n-1}\overline{H}_{n-1}},
\end{equation*}
yielding the following recurrence relation
\begin{equation*}
c_n-c_n^- = \frac{\underline{H}_n}{\underline{H}_{n-2}} = \frac{\underline{H}_n\underline{H}_{n-3}}{\underline{H}_{n-2}\underline{H}_{n-1}}  \frac{\underline{H}_{n-1}}{\underline{H}_{n-3}} = q_{n-1}p_n (c_{n-1}-c_{n-1}^-).
\end{equation*}
This gives
\begin{equation*}
c_n - c_n- = \prod_{k=1}^{n} q_{k-1}p_k.
\end{equation*}
Since $p_n$ is a relative position, $c_n^-$ is independent of $p_n$ and since $p_0=c_0=1$,
\begin{equation*}
\frac{\partial c_n}{\partial p_n} = \prod_{k=1}^{n-1} p_kq_k
\end{equation*}
Since for all $k$, $c_k$ only depends on $p_\ell$ for $\ell\leq k$, the jacobian of the mapping $\phi : \text{Int}\,M_n \rightarrow ]0,1[^n$ that associates to any sequence $(c_1,\hdots,c_n)$ in the interior of $M_n$ the corresponding canonical moments $(p_1,\hdots,p_n)$ is lower triangular, and its determinant is given by
\begin{equation*}
\left|\frac{\partial(c_1,\hdots,c_n)}{\partial (p_1,\hdots,p_n)}\right| = \prod_{k=1}^{n}\frac{\partial c_k}{\partial p_k} = \prod_{k=1}^{n-1}(p_kq_k)^{n-k}.
\end{equation*}
\end{proof}
Therefore the beta distribution naturally arises in the structure of the canonical moments space. Based on the Hankel determinants and similar quantities, the canonical moments can be computed from the regular moments using the so-called $Q$-$D$ algorithm. A description and proof and this algorithm can be found in \cite{dette1997}.

\subsection{Geometry of canonical moments}

Now we define a natural geometric structure for the space of canonical moments. The goal is to be able to manipulate probability distributions, e.g. compute distances or centroids, through this geometric structure in the representation space given by the canonical moments. We propose a construction that associates to each sequence of canonical moments $(p_1,\hdots,p_n)$ a point in the product beta manifold $(\mathcal B^n, g^n)$, where $g^n$ is the product of Fisher metrics on $\mathcal B$. More precisely, we associate to each $p_k$ a beta distribution with the corresponding mean value, i.e., a point of the parameter space $\Theta=\{(\alpha, \beta), \alpha>0, \beta>0\}$ belonging to $\Delta_{p_k}$, where $\Delta_p$ is the straight line of equation
\begin{equation*}
\Delta_{p} : \beta = \left(\frac{1}{p}-1\right) \alpha, \quad 0<p<1.
\end{equation*}
Let $0<p<1$, and $B\in \mathcal B$. We denote by $\phi : \mathcal B \times (0,1) \rightarrow \mathcal B$ the map that associates to $B$ and $p$ the closest neighbor of $B$ on $\Delta_p$
\begin{equation*}
\phi(B;p) = \text{argmin}\, \{ d(B, B'), B'\in\Delta_p\}.
\end{equation*}
On the basis of Theorem \ref{thm:betamoments}, we define the following mapping
\begin{equation*}
\Phi :\begin{cases} (0,1)^n \rightarrow \mathcal B^n\\
(p_1,\hdots,p_n) \mapsto \big(\phi(B(n,n);p_1),\hdots, \phi(B(1,1);p_n)\big),
\end{cases}
\end{equation*}
Notice that $\Phi$ associates the center $(1/2,\hdots,1/2)$ of the cube $[0,1]^n$ to the sequence $(B(n,n),\hdots,B(1,1))$, and associates any sequence $(p_1,\hdots,p_n)$ in $(0,1)^n$ to the sequence of beta distributions for which each component $B_k$ is the closest beta distribution to $B(n-k+1,n-k+1)$ with mean value $p_k$. 

Now let $\cc_n,\cc_n' \in\text{Int}\,M_n$. Then their canonical moments representations are in $(0,1)^n$, and we can define the following measure of dissimilarity
\begin{equation*}
\rho_n(\cc_n,\cc'_n) = d_n^F(\Phi(p_1,\hdots,p_n),\Phi(p_1',\hdots, p_n')), 
\end{equation*}
where $p_k = p_k(\cc_n)$ and $p_k'=p_k(\cc'_n)$, and
\begin{equation*}
d_n^F((B_1,\hdots,B_n),(B_1',\hdots,B_n')) = (d^F(B_1,B_1')^2+\hdots + d^F(B_n,B_n')^2)^{1/2}
\end{equation*}
is the product distance on $\mathcal B^n$. Since $\cc_n$ uniquely determines $(p_1,\hdots,p_n)$ and $\Phi$ is injective, $\rho_n$ is a distance. It is now possible to define e.g. the centroid of several moment sequences $\cc_n^{(k)}, k=1,\hdots,m$, as the Fréchet mean
\begin{equation*}
    \overline{\cc_n} = \underset{\cc_n}{\text{argmin}} \sum_{k=1}^m \rho_n^2(\cc_n,\cc_n^{(k)}),
\end{equation*}
and to compute it using a Karcher flow in the product manifold $\mathcal B^n$. Due to the negative curvature of that Riemannian manifold, we know that it will be unique.

\section{Conclusion and future work}\label{sec:conclusion}
The beta distribution is a natural exponential family, thus admitting a simple expression of the Fisher information metric as the Hessian of the log partition function. Nevertheless, explicit computations are difficult to conduct due to the presence of polygamma functions in the expression of the metric. In the present work, negative sectional curvature of the beta manifold was proved using recent results on ratios of polygamma functions and a lower bound on the determinant of the metric was obtained, that can further used to derive rates of expansion for geodesic balls. Finally, the relationship with canonical moments allows us to better understand them from a geometrical standpoint.

\bibliographystyle{plain}
\bibliography{bibliographie}

\begin{thebibliography}{10}

\bibitem{atkinson1981}
Colin Atkinson and Ann~FS Mitchell.
\newblock Rao's distance measure.
\newblock {\em Sankhy{\=a}: The Indian Journal of Statistics, Series A}, pages
  345--365, 1981.

\bibitem{bauer2016}
Martin Bauer, Martins Bruveris, and Peter~W Michor.
\newblock Uniqueness of the {F}isher--{R}ao metric on the space of smooth
  densities.
\newblock {\em Bulletin of the London Mathematical Society}, 48(3):499--506,
  2016.

\bibitem{cencov2000}
Nikolai~Nikolaevich Cencov.
\newblock {\em Statistical decision rules and optimal inference}.
\newblock Number~53. American Mathematical Soc., 2000.

\bibitem{dette1997}
Holger Dette and William~J Studden.
\newblock {\em The theory of canonical moments with applications in statistics,
  probability, and analysis}, volume 338.
\newblock John Wiley \& Sons, 1997.

\bibitem{fisher1922}
Ronald~A Fisher.
\newblock On the mathematical foundations of theoretical statistics.
\newblock {\em Philosophical Transactions of the Royal Society of London.
  Series A, Containing Papers of a Mathematical or Physical Character},
  222(594-604):309--368, 1922.

\bibitem{frechet1948}
M.~Fr{\'e}chet.
\newblock Les {\'e}l{\'e}ments al{\'e}atoires de nature quelconque dans un
  espace distanci{\'e}.
\newblock {\em Ann. Inst. H. Poincar{\'e}}, 10(3):215--310, 1948.

\bibitem{karcher1977}
H.~Karcher.
\newblock Riemannian center of mass and mollifier smoothing.
\newblock {\em Communications on pure and applied mathematics}, 30(5):509--541,
  1977.

\bibitem{olver2010nist}
F.W.J. Olver, National~Institute of~Standards, Technology (U.S.), D.W. Lozier,
  R.F. Boisvert, and C.W. Clark.
\newblock {\em NIST Handbook of Mathematical Functions Paperback and CD-ROM}.
\newblock Cambridge University Press, 2010.

\bibitem{rao1992}
C~Radhakrishna Rao.
\newblock Information and the accuracy attainable in the estimation of
  statistical parameters.
\newblock In {\em Breakthroughs in statistics}, pages 235--247. Springer, 1992.

\bibitem{schiff2013laplace}
J.L. Schiff.
\newblock {\em The Laplace Transform: Theory and Applications}.
\newblock Undergraduate Texts in Mathematics. Springer New York, 2013.

\bibitem{skovgaard1984}
Lene~Theil Skovgaard.
\newblock A {R}iemannian geometry of the multivariate normal model.
\newblock {\em Scandinavian Journal of Statistics}, pages 211--223, 1984.

\bibitem{yang2017}
Zhen-Hang Yang.
\newblock Some properties of the divided difference of psi and polygamma
  functions.
\newblock {\em Journal of Mathematical Analysis and Applications}, 455(1):761
  -- 777, 2017.

\end{thebibliography}


\end{document}